\documentclass[11pt]{article}
\usepackage{amssymb, amsmath, amsthm, amsfonts,xcolor, enumerate,graphicx}
\usepackage{epstopdf}
\usepackage{epsfig}
\usepackage{fullpage}
\usepackage{hyperref}
\usepackage{amsmath}
\usepackage{amsfonts}
\usepackage{amssymb}
\usepackage{graphicx}
\usepackage{amsthm}
\usepackage{enumerate}
\usepackage{setspace}
\usepackage{booktabs}
\usepackage{tabularx}
\usepackage{lineno} 
\usepackage{color} 
\usepackage{enumerate}
\usepackage{xcolor}

\usepackage{enumitem}
\usepackage{hyperref}
\usepackage{verbatim}
\usepackage{mathtools, amsmath, amsthm, amssymb, mathrsfs}

\usepackage{tikz}
\usetikzlibrary{arrows,shapes,snakes,automata,backgrounds,petri,positioning,shadows}

\tikzstyle{vertex}=[circle, draw, scale=0.5, fill=white] 

\newtheorem{theorem}{Theorem}[section]
\newtheorem{lemma}[theorem]{Lemma}

\newtheorem{defn}[theorem]{Definition}
\newtheorem{cor}[theorem]{Corollary}

\newtheorem{conj}[theorem]{Conjecture}
\newtheorem{claim}[theorem]{Claim}

\theoremstyle{definition}

\newtheorem{assump}{Assumption}

\usepackage{hyperref}

\def\komment#1{}
\let\komment=\footnote

\title{On the clique number of the square of a line graph and its relation to Ore-degree}

\author{
Maxime Faron
\thanks{Ecole Normale Superieure de Lyon, Lyon, France. Email: {\tt maxime.faron@ens-lyon.org}.}
\and
Luke Postle
\thanks{Department of Combinatorics and Optimization,  University of Waterloo. Email: {\tt lpostle@uwaterloo.ca}. Canada Research Chair in Graph Theory. Partially supported by NSERC under Discovery Grant No. 2014-06162, the Ontario Early Researcher Awards program and the Canada Research Chairs program.}}

\begin{document}
\maketitle

\begin{abstract}
\noindent In 1985, Erd\H{o}s and Ne\v{s}et\v{r}il conjectured that the square of the line graph of a graph $G$, that is $L(G)^2$, can be colored with $\frac{5}{4}\Delta(G)^2$ colors. This conjecture implies the weaker conjecture that the clique number of such a graph, that is $\omega(L(G)^2)$, is at most $\frac{5}{4}\Delta(G)^2$. In 2015, \'Sleszy\'nska-Nowak proved that $\omega(L(G)^2)\le \frac{3}{2}\Delta(G)^2$. In this paper, we prove that $\omega(L(G)^2)\le \frac{4}{3}\Delta(G)^2$. This theorem follows from our stronger result that $\omega(L(G)^2)\le \frac{\sigma(G)^2}{3}$ where $\sigma(G) := \max_{uv\in E(G)} d(u) + d(v)$, is the Ore-degree of the graph $G$. 
\end{abstract}

\section{Introduction}

The \emph{strong chromatic index}, $\chi'_s(G)$, of a graph $G$ is defined as the least integer $k$ for which there exists a $k$-coloring of $E(G)$ such that edges at distance at most $2$ receive different colors. Equivalently, $\chi'_s(G) = \chi (L(G)^2)$, where $L(G)^2$ denotes the square of the line graph of $G$. Since $\Delta(L(G)^2) < 2\Delta(G)^2$, the trivial upper bound on the chromatic number gives that $\chi'_s(G) \leq 2\Delta(G)^2$. However Erd\H{o}s and Ne\v{s}et\v{r}il (see~\cite{EN,FGHT1}) conjectured a much stronger upper bound as follows.

\begin{conj}\label{conj:EN1}
If $G$ is a graph, then $\chi'_s(G) \leq 1.25 \Delta(G)^2$.
\end{conj}

Note that if Conjecture~\ref{conj:EN1} is true, then the bound would be tight as the following example shows. Indeed, if $G_k$ denotes the graph obtained from a $5$-cycle by blowing up each vertex into an independent set of $k$ vertices, then $\Delta(G_k) = 2k$ and $L(G_k)^2$ is a clique with $5 k^2 = 1.25 \Delta(G_k)^2$ vertices. Figure~\ref{fig:blowup} depicts the graph $G_3$.

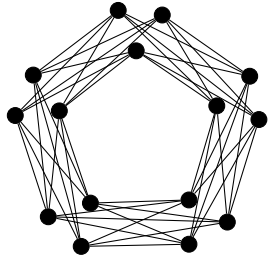
\begin{figure}[h]
\centering
\begin{tikzpicture}[scale=1.7]
\tikzstyle{whitenode}=[draw=white,circle,fill=white,minimum size=0pt,inner sep=0pt]
\tikzstyle{blacknode}=[draw,circle,fill=black,minimum size=6pt,inner sep=0pt]
\draw (0,0) node[whitenode] (a1) {};
\foreach \i in {2,3,4,5} {\pgfmathparse{\i-1} \let\j\pgfmathresult \draw (a\j) ++(72*\j-72:1) node[whitenode] (a\i) {};}
\foreach \i in {1,2,3,4,5} {\foreach \j in {1,2,3} {\draw (a\i) ++(120*\j+72*\i-2*72:0.2) node[blacknode] (b\i\j) {};}}
\foreach \i/\k in {1/2,2/3,3/4,4/5,5/1} {\foreach \j in {1,2,3} {\foreach \l in {1,2,3} {\draw (b\i\j) edge node {} (b\k\l);}}}
\end{tikzpicture}\caption{A blow-up of the $5$-cycle.}\label{fig:blowup}
\end{figure}

In 1997, Molloy and Reed~\cite{MR} made the first step towards Conjecture~\ref{conj:EN1}. They showed that for all graphs $G$, the graph $L(G)^2$ is $1/36$-sparse. Thus the naive coloring procedure guarantees that such a graph can be colored with $(1-\varepsilon)(\Delta(L(G)^2) + 1)$ colors for some $\varepsilon >0$. With their naive coloring procedure, the value of $\varepsilon$ that Molloy and Reed obtain is approximately $0.0238 \cdot \frac{1}{36} \approx 0.0007$. Bruhn and Joos~\cite{BJ} improved the bound on the neighborhood sparsity by showing that $L(G)^2$ is $1/4$-sparse. Combining this with an improved coloring procedure, they deduce that $\varepsilon \approx 0.0347$ suffices. Bonamy, Perrett and the second author~\cite{BPP} improved the bound even further by using an iterative version of the improved coloring procedure as well as an improved sparsity bound for critical graphs, proving $\chi'_s(G) \leq 1.835 \Delta^2$. 

Conjecture~\ref{conj:EN1} implies a bound on the clique number of $L(G)^2$. Since the tight example mentioned above is also a clique in $L(G)^2$, we have the following weaker conjecture by Faudree, Gy\'arf\'as, Schelp, and Tuza~\cite{FGHT2} from 1990 (see also~\cite{Bang}).

\begin{conj}\label{conj:EN2}
If $G$ is a graph, then $\omega(L(G)^2) \leq 1.25 \Delta(G)^2$.
\end{conj}

Some progress on Conjecture~\ref{conj:EN2} has already been made. First, the case of bipartite graphs has been solved in a stronger form by Faugree et. al.~\cite{FGHT2} since 1990 as follows.

	\begin{theorem}	\label{bip_delta}
		If $G$ is a bipartite graph, then $\omega(L(G)^2)\leq \Delta(G)^2$.
	\end{theorem}
		
Indeed, this bound is tight since the set of edges of a complete bipartite graph $K_{\Delta,\Delta}$ is a clique in $L(K_{\Delta,\Delta})^2$ of size $\Delta^2$. In a recent article, \'Sleszy\'nska-Nowak~\cite{Nowak} gave a new proof of theorem \ref{bip_delta}. Furthermore, she improved the bound for the general case as follows.

\begin{theorem}\label{gen_3_2_delta}
If $G$ is a graph, then $\omega(L(G)^2)\leq 1.5\Delta(G)^2$.
\end{theorem}

\subsection{Main Results}

In \cite{Nowak}, Theorem \ref{gen_3_2_delta} is proved by counting edges from a vertex $v$ of maximum degree. The crucial part of the argument is counting the number of edges which are not incident with any neighbor of $v$. This set is counted by means of vertex covers; however, the proof does not make use of the fact that this set of edges must also induce a clique in $L(G)^2$. If only induction could be applied to this set, then the bound might be improved; indeed, an unsophisticated inductive application of the result could be used to improve the bound to about $1.49\Delta(G)^2$. The trouble is that in this new graph the maximum degree might not decrease. Yet for any edge in the clique in $L(G)^2$, the sum of the degrees of its ends does decrease dramatically. This motivates then the use of Ore-degree, defined as follows, to more directly apply induction on said subgraph.

\begin{defn}[Ore-degree of a subgraph]
Let $G$ be a graph and $H$ be a subgraph of $G$. We define the \emph{Ore-degree} of $H$ in $G$ as $\sigma_G(H)=\max_{xy\in E(H)}(d_G(x)+d_G(y))$. Moreover, the \emph{Ore-degree} of $G$, which we denote by $\sigma(G)$, is defined as $\sigma_G(G)$.
\end{defn}

As will be seen later, Ore-degree is flexible enough to be used inductively in these kinds of proofs. Moreover, Ore-degree is perhaps the more natural parameter to bound $\omega(L(G)^2)$ since if $G$ is simple, $\sigma(G) = \Delta(L(G))+2$. In addition, there is also the natural question of whether analogues of the conjectures and theorems stated above hold for Ore-degree. For example, it is not immediate that the analogue of the trivial upper bound, that is $\omega(L(G)^2) \le .5 \sigma(G)^2$, is true; however, we found a short proof of this fact which we omit since this is implied by our stronger result Corollary~\ref{bip_H_sigma_H_2}.

Our first main result relates the clique number in the square of the line graph of a bipartite graph (even a multigraph) to Ore-degree as follows.

\begin{theorem}\label{bip_sigma_G}
If $G$ is a bipartite multigraph, then $\omega(L(G)^2) \leq \frac{1}{4}\sigma(G)^2$.
\end{theorem}

Theorem~\ref{bip_sigma_G} implies Theorem~\ref{bip_delta} since $\sigma(G)\le 2\Delta(G)$ and indeed Theorem~\ref{bip_sigma_G} is tight for the complete bipartite graph.

In fact, we prove a stronger result wherein we use only the Ore-degree of the clique instead of the whole graph as follows.

\begin{theorem}\label{bip_G_sigma_H}
If $G$ is a bipartite multigraph and $H$ is a subgraph of $G$ such that $E(H)$ is a clique in $L(G)^2$, then $|E(H)| \leq \Delta(H)(\sigma_G(H)-\Delta(H)) \leq \frac{1}{4}\sigma_G(H)^2$.
\end{theorem}

Theorem~\ref{bip_G_sigma_H} is more useful for inductive purposes then Theorem~\ref{bip_sigma_G} since we may only be able to control the Ore-degree of edges in $H$. Moreover, we conjecture that the same result holds when only $H$ is bipartite as follows.

\begin{conj}\label{bip_H_sigma_H_conj}
If $G$ is a graph, and $H$ is a bipartite subgraph of $G$ such that $E(H)$ is a clique in $L(G)^2$, then $|E(H)| \leq \frac{1}{4}\sigma_G(H)^2$.
\end{conj}


While we cannot prove Conjecture~\ref{bip_H_sigma_H_conj}, we can prove that it implies Conjecture~\ref{conj:EN2} as follows.

\begin{theorem}\label{reduction}
		Let $G$ be a graph, $H$ a subgraph of $G$ such that $E(H)$ is a clique in $L(G)^2$, and $a\in \left[\frac 1 4;\frac 1 3\right]$.

		If the following assumption holds:
		\begin{assump}
			\label{ass}
			For all bipartite subgraphs $H'$ of $H$ such that $|E(H')|<|E(H)|$, we have $|E(H')|\leq a \cdot \sigma_{G[V(H')]}(H')^2$.
		\end{assump}
		Then
		$$|E(H)|\leq  \left(\frac{1+a}{4}\right)\sigma_G(H)^2.$$
\end{theorem}

Note than that if Assumption~\ref{ass} holds with $a=\frac{1}{4}$, then by Theorem~\ref{reduction}, $\omega(L(G)^2)\le \frac{5}{16} \sigma_G(H)^2 \le \frac{5}{4} \Delta(G)^2,$ and hence Conjecture~\ref{bip_H_sigma_H_conj} implies Conjecture~\ref{conj:EN2}. Moreover, Assumption~\ref{ass} is itself true inductively for $a=\frac{1}{3}$ as the following corollary demonstrates.

	\begin{cor}\label{bip_H_sigma_H_2}
		If $G$ is a graph and $H$ is a subgraph of $G$ such that $E(H)$ is a clique in $L(G)^2$, then $|E(H)|\leq \frac{1}{3} \sigma_G(H)^2 \le \frac{1}{3} \sigma(G)^2$.
	\end{cor}

	\begin{proof}
		We proceed by induction on $|E(H)|$. If $|E(H)|\le 1$, then the the result follows trivially. So we may assume $|E(H)|\ge 2$. Let $H'$ be a bipartite subgraph of $H$ such that $|E(H')|<|E(H)|$. Now $E(H')$ is also a clique in $L(G)^2$ since $E(H)$ is. But then $E(H')$ is a clique in $L(G[V(H')])^2$. So by induction, $|E(H')|\le \frac{1}{3} \sigma_{G[V(H')]}(H')^2$. Thus Assumption \ref{ass} of Theorem $\ref{reduction}$ holds for $H$ with $a=\frac{1}{3}$. But then by Theorem \ref{reduction} with $a=\frac{1}{3}$, $|E(H)| \leq \left(\frac{1+\frac{1}{3}}{4}\right) \sigma_G(H)^2= \frac{1}{3} \sigma_G(H)^2$, as desired.
	\end{proof}

Since, $\sigma(G)\le 2\Delta(G)$, we have the following progress towards Conjecture~\ref{conj:EN2}.

	\begin{cor}\label{gen_4_3_delta}
		If $G$ is a graph, then $\omega(L(G)^2)\leq \frac{4}{3} \Delta(G)^2$.
	\end{cor}

In addition, we should mention that Reed~\cite{MR2} proved that $\chi_f(G)\le \left\lceil \frac{\Delta(G)+1+\omega(G)}{2} \right\rceil$ where $\chi_f(G)$ is the fractional chromatic number. Hence Corollary~\ref{gen_4_3_delta} implies that $\chi_f(L(G)^2) \le \frac{5}{3} \Delta(G)^2$ which is progress toward the fractional chromatic version of Conjecture~\ref{conj:EN1}.

Lastly, one may wonder if the results are tight. Corollary~\ref{gen_4_3_delta} is not if Conjecture~\ref{conj:EN2} is true. Yet when $G$ is bipartite, Theorem~\ref{bip_delta} is tight for the complete bipartite graph. Is this the only extremal example? Essentially yes as we prove the following stability version of Theorem~\ref{bip_delta}.

\begin{theorem}\label{stability}
For all $\varepsilon\in [0,1]$, if $G$ is a bipartite graph such that $\omega(L(G)^2)\geq (1-\varepsilon)\Delta(G)^2$, then $G$ contains a subgraph isomorphic to $K_{r,r}$ where $r=(1-\sqrt{8}\varepsilon^{1/4})\Delta(G)$.
\end{theorem}

The paper is organized as follows. In Section~\ref{sec:bip}, we prove Theorem~\ref{bip_G_sigma_H}. In Section~\ref{sec:gen}, we prove Theorem~\ref{reduction}. Finally in Section~\ref{sec:stab}, we prove Theorem~\ref{stability}.

\section{Proof of the Bipartite Result}\label{sec:bip}

		\begin{proof}[Proof of Theorem~\ref{bip_G_sigma_H}.]

We may assume $E(H)\ne\emptyset$ as otherwise the result follows trivially. For ease of reading, let us denote $\Delta(H)$ by $\Delta_H$ and $\sigma_G(H)$ by $\sigma$.	Let $v$ be a vertex of maximum degree in $H$, that is $d_H(v)=\Delta_H$. 

We may assume without loss of generality that $V(G)=V(H)$ since $E(H)$ is a clique in $L(G[V(H)])^2$. Moreover, since $E(H)$ is a clique in $L(G)^2$, all edges of $H$ must be at distance at most 2 in $G$ of each edge in $H$ incident to $v$. Thus each vertex in $H$ must be at distance at most 3 in $G$ of $v$. Since $V(G)=V(H)$, we have that every vertex in $G$ is at distance at most 3 in $G$ of $v$.

Let $A=N_H(v)$, $C=N_G(v)\setminus A$ and finally let $S=\{u\in V(H): d_G(u,v)=2, \exists uw\in E(H), d_G(v,w)=3\}$. Let $E_C$ be the set of edges of $H$ incident with a vertex of $C$, $E_S$ be the set of edges of $H$ incident with a vertex of $S$ and finally let $E_A$ be the set of edges of $H$ incident with a vertex of $A$ but not a vertex in $S$ (see Figure~\ref{schema_proof_bip_G_sigma_H} for an illustration). 

Note that every edge $e$ of $H$ that is not incident with a vertex in $N_G(v)$ must be incident with a vertex $u$ such that $d_G(u,v)=2$ since $E(H)$ is a clique in $L(G)^2$. Since $G$ is bipartite, the other end, call it $w$, of such an edge must be odd distance from $v$ and hence $d_G(v,w)=3$. But then $u\in S$ by definition. That is, every edge of $H$ is incident with a vertex in $A$, $C$ or $S$. Thus $E(H)= E_C\cup E_S\cup E_A$.
 
			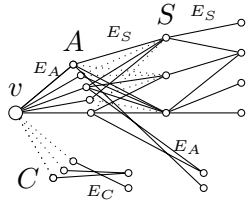
\begin{figure}[h]
				\begin{center}

					\begin{tikzpicture}[node distance=2cm, on grid, auto]
						\node[vertex, label=$v$] (A) {};

						\draw[thin] (A) -- (5*10-30+20:1cm) node[vertex, scale=0.5, label=$A$] (A-5) {}   node[midway, above] {\tiny$E_A$};
						\foreach \x in {1,...,5}
						{
							\draw[thin] (A) -- (\x*10-30+20:1cm) node[vertex, scale=0.5] (A-\x) {};
						}
						\draw[thin, dotted] (A) -- ++ (1*10-30-40:1cm) node[vertex, scale=0.5, solid, label=left:$C$] (C-1) {};
						\foreach \x in {1,...,3}
						{
							\draw[thin, dotted] (A) -- ++ (\x*10-30-40:1cm) node[vertex, scale=0.5, solid] (C-\x) {};
						}

						\node[vertex, label=above:$S$, scale=0.5] (S-1) at (2,1) {};
						\node[vertex, scale=0.5] (S-2) at (2,0.5) {};
						\node[vertex, scale=0.5] (S-3) at (2,0) {};

						\draw[thin, dotted] (A-1) -- (S-1);
						\draw[thin]	(A-1) -- (S-2);
						\draw[thin]	(A-1) -- (S-3);

						\draw[thin] (A-2) -- (S-1);
						\draw[thin, dotted]	(A-2) -- (S-2);
						\draw[thin, dotted]	(A-2) -- (S-3);

						\draw[thin] (A-3) -- (S-1);
						\draw[thin]	(A-3) -- (S-2);
						\draw[thin]	(A-3) -- (S-3);

						\draw[thin, dotted] (A-4) -- (S-1);
						\draw[thin, dotted]	(A-4) -- (S-2);
						\draw[thin]	(A-4) -- (S-3);

						\draw[thin] (A-5) -- (S-1) node[midway, above] {\tiny$E_S$};
						\draw[thin, dotted]	(A-5) -- (S-2);
						\draw[thin]	(A-5) -- (S-3);

						\node[vertex, scale=0.5] (D-1) at (3,1.2) {};
						\node[vertex, scale=0.5] (D-2) at (3,0.8) {};
						\node[vertex, scale=0.5] (D-3) at (3,0.4) {};
						\node[vertex, scale=0.5] (D-4) at (3,0) {};
						\draw[thin]
							(S-1) -- (D-1) node[midway, above] {\tiny$E_S$}
							(S-1) -- (D-2)
							(S-2) -- (D-2)
							(S-3) -- (D-2)
							(S-3) -- (D-3)
							(S-3) -- (D-4);

						\node[vertex, scale=0.5] (B-1) at (2.5,-0.8) {};
						\node[vertex, scale=0.5] (B-2) at (2.5,-1) {};

						\draw[thin]
							(A-1) -- (B-1) node[very near end, above] {\tiny$E_A$}
							(A-3) -- (B-1)
							(A-5) -- (B-2);

						\node[vertex, scale=0.5] (C'-1) at (1.5,-0.8) {};
						\node[vertex, scale=0.5] (C'-2) at (1.5,-1) {};

						\draw[thin]
							(C-1) -- (C'-1)
							(C-2) -- (C'-1)
							(C-3) -- (C'-2) node[midway, below] {\tiny$E_C$};

					\end{tikzpicture}
					\caption{Diagram of notations from the proof of theorem \ref{bip_G_sigma_H}.}
					\label{schema_proof_bip_G_sigma_H}
				\end{center}
			\end{figure}

			We claim that every vertex $u\in S$ is adjacent to all vertices in $A$. To see this recall that by definition of $S$, there exists $uw\in E(H), d_G(v,w)=3$. But for all $x\in A$, $d_G(vx,uw)=2$ and hence $xu\in E(G)$ as claimed since none of the other edges $vu,vw,xw$ exist in $G$ given the distances of $x$, $u$, and $w$ to $v$ in $G$.

			Now $|E_C| \le |C|\Delta_H = (d_G(v)-|A|)\Delta_H$ and $|E_S|\le |S|\Delta_H$. Meanwhile, for every $x\in A$, there are at most $d_G(x)-|S|$ edges in $E_A$ that are incident to $x$ since every vertex in $A$ is adjacent to all vertices in $S$ and such edges are not in $E_A$. Thus $|E_A| \le \sum_{x\in A} (d_G(x)-|S|)$. Yet for all $x\in A$, $d_G(x)+d_G(v) \le \sigma$. Hence $|E_A| \le |A|(\sigma-d_G(v)-|S|)$.
 
			So we have

			$$\begin{aligned}|E(H)|&\leq |E_C| + |E_A| + |E_S|\\
					&\leq (d_G(v)-|A|)\Delta_H + |A|(\sigma - d_G(v) -|S|) + |S|\Delta_H.\end{aligned}$$

			However, $|A|=\Delta_H$ and hence
			
			$$|E(H)|\leq \Delta_H(\sigma -\Delta_H) \leq \left(\frac {\sigma} {2}\right)^2.$$
		\end{proof}

\section{Proof of the General Result}\label{sec:gen}

To prove Theorem~\ref{reduction}, we decompose the graph $H$ into several sets of edges and count such sets in two different ways, one involving Assumption~\ref{ass} and the other involving trivial bounds. 

	\begin{proof}[Proof of Theorem~\ref{reduction}.]

We may assume without loss of generality that $V(G)=V(H)$ since $E(H)$ is also a clique in $L(G[V(H)])^2$. Furthermore, we may assume $E(H)\ne \emptyset$ as otherwise the result follows trivially. For ease of reading, we let $d(v)$ denote the degree in $H$ of a vertex $v$, that is $d_H(v)$, and we also let $\sigma$ denote $\sigma_G(H)$. 


Let $x$ be a vertex of maximum degree in $H$ and let $y$ be a neighbor in $H$ of $x$ such that $d_H(y)$ is maximized. 

Note also that since $xy\in E(H)$, every vertex of $H$ (and hence also of $G$) is at distance at most two in $G$ from at least one of $x$ or $y$. Thus we let $A=N_G(x)\cup N_G(y) \setminus\{x,y\}$, that is $A=\{z\in V(H): d_G(z,\{x,y\})=1\}$ and we let $B=\{z\in V(H): d_G(z,\{x,y\})=2\}$. So $V(G)=\{x,y\}\cup A\cup B$. 

Let us further partition $A$ as follows. Let $A_1=N_G(x)\setminus N_G(y)$, $A_2=N_G(x)\cap N_G(y)$ and $A_3=N_G(y)\setminus N_G(x)$, see Figure \ref{schema_proof_reduction} for an illustration. 

		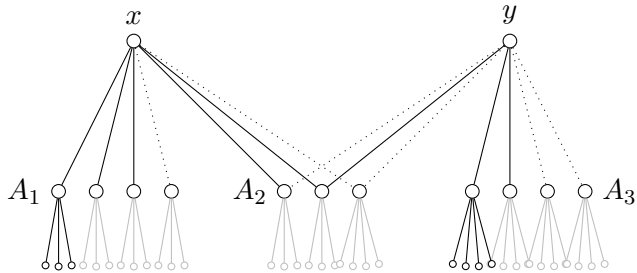
\begin{figure}[h]
			\begin{center}
					\begin{tikzpicture}[node distance=2cm, on grid, auto]
						\node[vertex, label=$x$] (x) {};
						\node[vertex, label=$y$] (y) at (5,0) {};

						\node[vertex, label=left:$A_1$] (A1-1) at (-1,-2) {};
						\node[vertex] (A1-2) at (-0.5,-2) {};
						\node[vertex] (A1-3) at (0,-2) {};
						\node[vertex] (A1-4) at (0.5,-2) {};

						\node[vertex, label=left:$A_2$] (A2-1) at (2,-2) {};
						\node[vertex] (A2-2) at (2.5,-2) {};
						\node[vertex] (A2-3) at (3,-2) {};

						\node[vertex] (A3-1) at (4.5,-2) {};
						\node[vertex] (A3-2) at (5,-2) {};
						\node[vertex] (A3-3) at (5.5,-2) {};
						\node[vertex, label=right:$A_3$] (A3-4) at (6,-2) {};

						\draw[thin] (x) -- (A1-1);
						\draw[thin] (x) -- (A1-2);
						\draw[thin] (x) -- (A1-3);
						\draw[thin, dotted] (x) -- (A1-4);
						\draw[thin] (x) -- (A2-1);
						\draw[thin] (x) -- (A2-2);
						\draw[thin, dotted] (x) -- (A2-3);

						\draw[thin, dotted] (y) -- (A2-1);
						\draw[thin] (y) -- (A2-2);
						\draw[thin, dotted] (y) -- (A2-3);
						\draw[thin] (y) -- (A3-1);
						\draw[thin] (y) -- (A3-2);
						\draw[thin, dotted] (y) -- (A3-3);
						\draw[thin, dotted] (y) -- (A3-4);

						\foreach \x in {1,...,3}
						{
							\draw[thin] (A1-1) -- ++ (\x*10-20-90:1cm) node[vertex, scale=0.5] () {};
							\draw[thin, draw=gray!50] (A1-2) -- ++ (\x*10-20-90:1cm) node[vertex, scale=0.5] () {};
							\draw[thin, draw=gray!50] (A1-3) -- ++ (\x*10-20-90:1cm) node[vertex, scale=0.5] () {};
							\draw[thin, draw=gray!50] (A1-4) -- ++ (\x*10-20-90:1cm) node[vertex, scale=0.5] () {};

							\draw[thin, draw=gray!50] (A2-1) -- ++ (\x*10-20-90:1cm) node[vertex, scale=0.5] () {};
							\draw[thin, draw=gray!50] (A2-2) -- ++ (\x*10-20-90:1cm) node[vertex, scale=0.5] () {};
						}

						\foreach \x in {1,...,4}
						{
							\draw[thin, draw=gray!50] (A2-3) -- ++ (\x*10-25-90:1cm) node[vertex, scale=0.5] () {};
							\draw[thin] (A3-1) -- ++ (\x*10-25-90:1cm) node[vertex, scale=0.5] () {};
							\draw[thin, draw=gray!50] (A3-2) -- ++ (\x*10-25-90:1cm) node[vertex, scale=0.5] () {};
							\draw[thin, draw=gray!50] (A3-3) -- ++ (\x*10-25-90:1cm) node[vertex, scale=0.5] () {};
							\draw[thin, draw=gray!50] (A3-4) -- ++ (\x*10-25-90:1cm) node[vertex, scale=0.5] () {};
						}

					\end{tikzpicture}
				\caption{Diagram of notations from the proof of Theorem \ref{reduction}.}
				\label{schema_proof_reduction}
			\end{center}
		\end{figure}

Finally it will be convenient to consider the neighbors in $H$ of $x$, so let $C=N_H(x)\setminus \{y\}$. Note that $C\subseteq A_1\cup A_2$. Moreover, for every vertex $v\in C$, $d(v)\le \sigma-d(x)$ and yet also $d(v)\le d(y)$ since $d(y)$ is maximized over $N_H(x)$.

Let $E_A$ be the set of edges with both ends in $A$. Note that every edge of $H$ other than $xy$ is incident with a vertex in $A$. 
Hence $|E(H)| = 1+\sum_{v\in A} d(v) - |E_A|$. Yet $\sum_{v\in C} d(v) \le |C| \cdot d(y)$ while $\sum_{v\in A\setminus C} d(v) \le |A\setminus C|\cdot d(x)$. Thus we have the following trivial bound on $|E(H)|$: 

		$$|E(H)|\leq 1 + |C|\cdot d(y) + |A\setminus C|\cdot d(x) - |E_A|.$$
		
Yet $|C|=d(x)$ and $|A|\le \sigma-|A_2|-2$. Thus we have: 

$$\begin{aligned}|E(H)|\leq& 1 + d(x) \cdot d(y) + (\sigma-|A_2|-2-d(x)) d(x) -|E_A|\\
			\leq& 1 -2d(x) + d(x)\sigma + d(x)d(y)-d(x)^2 -|A_2| d(x)-|E_A|.
		\end{aligned}$$

		We deduce the following simpler bound from the bound above:
		$$|E(H)|\leq d(x)(\sigma-d(x) + d(y)).$$

Now let us define two bipartite subgraphs of $H$ as follows: let $H_1$ be the graph such that $V(H_1)=A_1\cup B$ and $E(H_1)=\{uv\in E(H): u\in A_1, v\in B\}$; let $H_2$ be the graph such that $V(H_2)=A_3\cup B$ and $E(H_2)=\{uv\in E(H): u\in A_3, v\in B\}$. Since $E(H)$ is a clique in $L(G)^2$, so are $E(H_1)$ and $E(H_2)$. Furthermore, $E(H_1)$ is a clique in $L(G_1)^2$ and $E(H_2)$ is a clique in $L(G_2)^2$, where we let $G_1=G[V(H_1)]$ and $G_2=G[V(H_2)]$. 

In addition, for every edge $uv\in E(H_1)$ and $w\in N_H(y)$, either $u$ or $v$ must be adjacent to $w$. This implies that if $uv\in E(H_1)$, then $d_{G_1}(u) + d_{G_1}(v) \leq \sigma - d(y)$. So $\sigma_{G_1}(H_1)\leq \sigma - d(y)$. Thus by Assumption \ref{ass} applied to $H_1$, we find that 

$$|E(H_1)| \leq a (\sigma - d(y))^2.$$ 

Similarly $\sigma_{G_2}(H_2)\leq \sigma - d(x)$, and by Assumption \ref{ass} applied to $H_2$, we have 

$$|E(H_2)| \leq a (\sigma - d(x))^2.$$

Now every edge of $H$ is either incident to one of $x$, $y$ or a vertex of $A_2$, or is in one of $E(H_1), E(H_2)$ or $E_A$. Thus we get a new bound as follows:
$$|E(H)|\leq d(x)+d(y)-1 + |E_A| + |E(H_1)|+|E(H_2)| + |A_2|\cdot d(x).$$

Substituting the bounds for $|E(H_1)|$ and $|E(H_2)|$ from above now gives:
$$|E(H)| \leq d(x)+d(y)-1 + |E_A| + a (\sigma - d(y))^2 + a (\sigma - d(x))^2 + |A_2| \cdot d(x).$$
		
Then, the sum of the bound above and our first trivial bound is also a bound as follows:
$$2|E(H)| \leq d(y)-d(x) + (1-2a)d(x)\sigma -2a d(y)\sigma - (1-a) d(x)^2 + d(x)d(y) + 2a\sigma^2 + a d(y)^2.$$

Factoring out $d(x)$ and recalling that $d(y)\le d(x)$ gives
$$|E(H)|\leq \frac{1} {2} \left( d(x)\left( (1-2a) \sigma + d(y) - (1-a) d(x)\right) -2a d(y)\sigma + 2a\sigma^2 + a d(y)^2 \right).$$


		Recall that $a\in \left[\frac 1 4 ; \frac 1 3\right]$. Now we have two bounds, a simple one and an average one as follows:

		\begin{equation}
			|E(H)|\leq d(x)(\sigma-d(x) + d(y)).
			\label{simple}
		\end{equation}

		\begin{equation}
			|E(H)| \leq \frac{1} {2} \left( d(x)\left( (1-2a) \sigma + d(y) - (1-a) d(x) \right) -2a \cdot d(y)\sigma + 2a\sigma^2 + a d(y)^2 \right).
			\label{average}
		\end{equation}

		Next we set $s=\sqrt{1+a}-1$.
		We now distinguish two cases, depending on whether $\frac{d(y)}{\sigma}$ is more or less than $s$.\\

		\noindent{\bf Case 1: $d(y)\leq s \sigma.$} 
		
		Then, by (\ref{simple}), $|E(H)| \leq d(x)(\sigma-d(x) + d(y))$. Since $d(y)\le s\sigma$, we have that $$|E(H)|\le d(x)((1+s)\sigma-d(x)),$$ 
		which is at most $$(\frac{1+s}{2}\sigma)^2 = \frac{1+a}{4}\sigma,$$ as desired.\\
				
			\noindent{\bf Case 2: $d(y) \geq s \sigma$.}

				By (\ref{average}),

				$\begin{aligned}
					|E(H)| \leq& \frac{1} {2} \left( d(x)( (1-2a) \sigma + d(y) - (1-a) d(x)) -2a d(y)\sigma + 2a\sigma^2 + a d(y)^2 \right)\\
						\leq& \frac{1} {2} \left( (1-a) d(x)\left( \frac{(1-2a) \sigma + d(y)} {1-a} - d(x)\right) -2a d(y)\sigma + 2a\sigma^2 + a d(y)^2 \right).\\
				\end{aligned}$

				We would like to say that the right side of the inequality is maximized when $d(x)=\frac{(1-2a) \sigma + d(y)} {2(1-a)}$ but we should first distinguish whether or not $\frac{(1-2a) \sigma + d(y)} {2(1-a)}$ is greater than $\sigma - d(y)$ which is an upper bound on $d(x)$.\\

			\noindent{ \bf Case 2.1: $\frac{(1-2a) \sigma + d(y)} {2(1-a)} \leq \sigma-d(y)$, that is $d(y)\leq \frac{\sigma}{3-2a}$.}
			
				Then (\ref{average}) is maximized when $d(x)=\frac{(1-2a) \sigma + d(y)} {2(1-a)}$, whence we get

				$\begin{aligned}
					|E(H)| &\leq \frac{1} {2} \left( (1-a) \left( \frac{(1-2a) \sigma + d(y)} {2(1-a)}\right)^2 -2a d(y)\sigma + 2a\sigma^2 + a d(y)^2 \right)\\
						&\leq \frac{1} {8(1-a)} \left( (1 +4a - 4 a^2)\sigma^2 + (2-12a+8a^2) d(y)\sigma + (1+4a-4a^2) d(y)^2\right).\\
				\end{aligned}$

				Let 
				
				$$\begin{aligned} f(t) &= \frac{1} {8(1-a)} \left( (1 +4a - 4 a^2)\sigma^2 + (2-12a+8a^2) t\sigma + (1+4a-4a^2) t^2\right)\\
															 &= \frac{1}{8(1-a)} \left( (1+4a-4a^2)(\sigma+t)^2 + 4(4a^2-5a)t\sigma \right).\end{aligned}$$
															
			  Then $f$ is a second-degree polynomial, whose leading coefficient is $\frac{1+4a-4a^2}{8(1-a)}$. But, as $0 < a \le 1$, $\frac{1+4a-4a^2}{8(1-a)}\geq \frac{1}{8(1-a)} > 0$. So $f$ is a convex function. Hence

				$$\max_{t\in\left[s\sigma,\frac{\sigma}{3-2a}\right]} \left(f(t)\right) =\max\left(f(s\sigma),f\left(\frac{\sigma}{3-2a}\right)\right).$$\\
								
				\begin{claim}\label{First} 
				$f(s\sigma) \le \frac{1+a}{4}\sigma^2$.
				\end{claim}
				\begin{proof} Note that $4a^2-5a \leq 0$ since $a\leq 5/4$. Meanwhile $s = \sqrt{a+1}-1 \geq a/2 - a^2/8$ since $a+1 \ge (1+a/2 - a^2/8)^2 = 1+a - a^3/8 + a^4/64$ since $0\leq a\leq 8$. Thus
				
				$$\begin{aligned} f(s\sigma) &= \frac{1}{8(1-a)} \left( (1+4a-4a^2)(1+s)^2\sigma^2 + 4(4a^2-5a)s\sigma^2 \right)\\
																		 &\leq \frac{\sigma^2}{8(1-a)} \left( (1+4a-4a^2)(a+1) + 4(4a^2-5a)(a/2 - a^2/8) \right)\\
																		 &= \frac{\sigma^2}{8(1-a)}\left( 1 + 5a - 10a^2 + 6.5 a^3 - 2a^4  \right). \end{aligned}$$								
				
				However, 
				
				$$\begin{aligned} &\frac{1+a}{4} - \frac{1+5a-10a^2+6.5a^3-2a^4}{8(1-a)} \\
													= & \frac{2-2a^2 - (1+5a-10a^2+6.5a^3-2a^4)}{8(1-a)}\\
													= & \frac{1-5a+8a^2-6.5a^3 +2a^4}{8(1-a)}\\
													= & \frac{(1-3a)(1-2a+2a^2-.5a^3)+.5a^4}{8(1-a)}.\end{aligned}$$
													
			Yet $1-3a\ge 0$ as $a\le 1/3$, $1-a\ge 0$ as $a\le 1$, and $1-2a+2a^2-.5a^3 = 1 - 2a(1-a)-.5a^3 \ge 1-2(.5)^2 - .5(1)^3 = 0$ as $0\le a \le 1$. Thus the difference in the equation above is at least $0$ and hence $\frac{1+a}{4}\sigma^2$ is at least $f(s\sigma)$ as desired.
			\end{proof}
							
			\begin{claim}\label{Second}
			$f(\frac{\sigma}{3-2a}) \le \frac{1+a}{4}\sigma^2$.
			\end{claim}
			\begin{proof} Note that 
			
			$$f\left(\frac{\sigma}{3-2a}\right)= \frac{-16a^4+48a^3-36a^2-12a+16}{8(1-a)(3-2a)^2}\sigma^2.$$ 
			
			But $-16a^4+48a^3-36a^2-12a+16 = 4(1-a)(4a^3-8a^2+a+4)\sigma^2$ and hence 
			
			$$f\left(\frac{\sigma}{3-2a}\right) = \frac{4a^3-8a^2+a+4}{2(3-2a)^2}\sigma^2.$$ 
			
			However, 
				
				$$\begin{aligned} &\frac{1+a}{4} - \frac{4a^3-8a^2+a+4}{2(3-2a)^2} \\
				                 = & \frac{(1+a)(3-2a)^2 - 2(4a^3-8a^2+a+4)}{4(3-2a)^2}\\
												 = & \frac{(9-3a-8a^2+4a^3)-(8a^3-16a^2+2a+8)}{4(3-2a)^2}\\
												 = & \frac{1-5a+8a^2-4a^3}{4(3-2a)^2}\\
												 = & \frac{(1-a)(2a-1)^2}{4(3-2a)^2}. \end{aligned}$$
				
				Thus since $a\le 1$, this is at least $0$. Hence $\frac{1+a}{4}\sigma^2$ is at least $f\left(\frac{\sigma}{3-2a}\right)$, as desired.
				\end{proof}
					
        Thus $|E(H)|\le \max\left(f(s\sigma),f\left(\frac{\sigma}{3-2a}\right)\right)$, which by Claims~\ref{First} and~\ref{Second} is at most $\frac{1+a}{4}\sigma^2$ as desired.\\
													
		\noindent{\bf Case 2.2: $\frac{(1-2a) \sigma + d(y)} {2(1-a)} > \sigma-d(y)$, that is $d(y)\geq \frac{\sigma}{3-2a}$.}
		
			By (\ref{average}), $$|E(H)|\leq \frac{1} {2} \left( (1-a) d(x)\left( \frac{(1-2a) \sigma + d(y)} {1-a} - d(x)\right) -2ad(y)\sigma + 2a\sigma^2 + a d(y)^2 \right).$$

			Let 
			
			$$g(t)=\left( \frac{(1-2a) \sigma + d(y)} {1-a} - t\right)t=\frac{(1-2a) \sigma + d(y)} {1-a} t - t^2.$$ 
			
			So $g$ is a second-degree polynomial whose leading coefficient is $-1\leq 0$ and hence is concave. Yet $g$ is maximized when $t=\frac{(1-2a) \sigma + d(y)} {2(1-a)}$ which is greater than $\sigma-d(y)$. So $g$ is an increasing function on $[0,\sigma-d(y)]$ and hence 
			
			$$\max_{t\in[0,\sigma-d(y)]} g(t)=g(\sigma-d(y)).$$

			Thus
			
			$\begin{aligned}
				|E(H)|\leq& \frac{1} {2} \left( (1-a) (\sigma-d(y))\left( \frac{(1-2a) \sigma + d(y)} {1-a} - (\sigma-d(y))\right) -2a d(y)\sigma + 2a\sigma^2 + a d(y)^2 \right) \\
					\leq& \frac{1} {2} \left( a\sigma^2- 2(1-a) d(y)^2 + 2(1-a) d(y) \sigma\right) \\
					\leq& \frac{1} {2} \left( a\sigma^2 + 2(1-a) d(y)(\sigma-d(y))\right) \\
					\leq& \frac{1} {2} \left( a\sigma^2 + 2(1-a) \frac{\sigma^2}{4}\right) \\
					\leq& \frac{1} {8} \left( (4a+2-2a)\sigma^2\right) \\
					\leq& \frac{1+a}{4} \sigma^2.
			\end{aligned}$
		










	\end{proof}

\section{Proof of the Stability Result}\label{sec:stab}

We prove Theorem~\ref{stability} in two parts. First we prove that if there is a large clique in $L(G)^2$, then there are two sets of size at most $\Delta$ with many edges between them as follows.

\begin{lemma}\label{stability1}
Let $\varepsilon\in [0,1]$. If $G=(A,B)$ is a bipartite graph and $H$ is a subgraph of $G$ such that $E(H)$ is a clique in $L(G)^2$ and $|E(H)|\geq (1-\varepsilon)\Delta(G)^2$, then there exists $A'\subseteq A, B'\subseteq B$ such that $|A'|,|B'|\leq \Delta(G)$ and $|E(H)\cap E(A',B')|\geq (1-2 \varepsilon - 2 \sqrt{\varepsilon})\Delta(G)^2$.
\end{lemma}
	\begin{proof}
	  Let $a$ be a vertex of $A$ such that $d_H(a)$ is maximum over all vertices in $A$. Let $S_a = \{u\in A: N_H(u)\setminus N_G(a) \ne \emptyset\}$. Note that every vertex of $S_a$ is adjacent in $G$ to every vertex of $N_H(a)$ since otherwise $E(H)$ is not a clique in $L(G)^2$. Let $E_a$ be the set of edges of $H$ not incident with a vertex in $N_G(a)$. Note that $|E_a| \le |S_a|(\Delta(G)-|N_H(a)|)$. Yet $|S_a|\le \Delta(G)$ since $S_a$ is contained in the neighborhood of any vertex in $N_H(a)$. Thus $|E_a| \le \Delta(G) (\Delta(G)-|N_H(a)|)$.
		
		Similarly let $b$ be a vertex of $B$ such that $d_H(b)$ is maximum over all vertices in $B$. Let $S_b = \{v\in B: N_H(v)\setminus N_G(b) \ne \emptyset\}$ and let $E_b$ be the set of edges of $H$ not incident with a vertex in $N_G(b)$. A symmetric argument to the one above shows that $|E_b| \le \Delta(G) (\Delta(G)-|N_H(b)|)$. 
		
		Let $B'=N_G(a)$ and $A'=N_G(b)$. Note that $|A'|,|B'|\le \Delta(G)$. Moreover, $|E(H)\cap E(A',B')| \ge |E(H)|-|E_a|-|E_b| \ge |E(H)| - \Delta(G)(2\Delta(G) - |N_H(a)| - |N_H(b)|)$.
		
	  Now we may assume without loss of generality that $d_H(a)\ge d_H(b)$. Thus $d_H(a)=\Delta(H)$. By Theorem \ref{bip_G_sigma_H}, $|E(H)| \leq \Delta(H)(2\Delta(G)-\Delta(H)|$. Thus $\Delta(H) \ge (1-\sqrt{\varepsilon})\Delta(G)$, for otherwise $|E(H)| < (1-\sqrt{\varepsilon})(1+\sqrt{\varepsilon})\Delta(G)^2 = (1-\varepsilon)\Delta(G)^2$, contrary to our assumption.
		
		But then $|E_a| \le \sqrt{\varepsilon}\Delta(G)^2$. So the number of edges of $H$ with one end in $B'$ is at least $(1-\varepsilon-\sqrt{\varepsilon})\Delta(G)^2$. Thus there must exist a vertex $v$ in $B'$ such that $|N_H(v)| \ge (1-\varepsilon-\sqrt{\varepsilon})\Delta(G)$ since $|B'|\le \Delta(G)$. But then $|N_H(b)|\ge |N_H(v)| \ge (1-\varepsilon-\sqrt{\varepsilon})\Delta(G)$ and hence $|E_b| \le (\varepsilon + \sqrt{\varepsilon})\Delta(G)^2$. Thus 
		
		$$|E(H)\cap E(A',B')| \ge |E(H)|-|E_a|-|E_b| \ge (1-2\varepsilon - 2\sqrt{\varepsilon})\Delta(G)^2,$$
		
		as desired.
	
	\end{proof}

Yet we can also prove that two such sets with many edges between them must contain a large complete bipartite subgraph as follows.

\begin{lemma}\label{stability2}
Let $\alpha\in [0,1]$. If $G=(A,B)$ is a bipartite graph with $|A|,|B|\leq n$ and $H$ is a subgraph of $G$ such that $|E(H)|\geq (1-\alpha)n^2$ and $E(H)$ is a clique in $L(G)^2$, then there exists a subgraph $J$ of $G$ isomorphic to $K_{r,r}$ where $r=(1-\sqrt{2\alpha})n$. 
\end{lemma}
\begin{proof}
We may assume without loss of generality that $|A|=|B|=n$ for otherwise we may add isolated vertices to $A$ and $B$ while maintaining the hypotheses of the lemma. Now let $G'$ be the graph such that $V(G')=(A,B)$ and $E(G')=\{ab\notin E(G): a\in A, b\in B\}$. Let $C$ be a minimum vertex cover of $G'$ and $M$ be a maximum matching of $G'$, say of size $m$. By K\"onig's theorem, $|V(C)|=m$. 

If $e_1=a_1b_1,e_2=a_2b_2$ are two distinct edges of $M$, then at least one of $a_1b_2,a_2b_1$ must not be an edge in $H$ as otherwise $E(H)$ is not a clique in $L(G)^2$ since the distance between $a_1b_2$ and $a_2b_1$ is at least $3$ in $G$. But there are at least $m(m-1)/2$ such pairs. So $|E(H)|\le |A||B|-m-\frac{m(m-1)}{2} \le n^2 - m^2/2$. Since $|E(H)|\ge (1-\alpha)n^2$ we find that $m\le \sqrt{2\alpha} n$. Yet $|V(C)|=m$. 

Now $G[V(G)-V(C)]$ is a complete bipartite graph since every edge of $G'$ is incident with a vertex of $C$. Let $A'\subseteq A-V(C)$ of size $r$ and $B'\subseteq B-V(C)$ of size $r$, which is possible since $m \le n-r$. Let $J=G[A'\cup B']$. Now $J$ is isomorphic to $K_{r,r}$ as desired.  
\end{proof}

We are now ready to prove Theorem~\ref{stability} using the two lemmas above.

\begin{proof}[Proof of Theorem~\ref{stability}.]
Let $G=(A,B)$ be a bipartition of $G$. Let $H$ be a subgraph of $G$ such that $E(H)$ is a clique in $L(G)^2$ and $|E(H)|\geq (1-\varepsilon)\Delta(G)^2$. By Lemma~\ref{stability1}, there exist $A'\subseteq A, B'\subseteq B$ such that $|A'|,|B'|\leq \Delta(G)$ and $|E(H)\cap E(A',B')|\geq (1-2 \varepsilon - 2 \sqrt{\varepsilon})\Delta(G)^2$. 

Let $G'=G[A'\cup B']$ and let $H'= H[A'\cup B']$. Now $H'$ is a subgraph of $G'$. Moreover, $E(H')$ is a clique in $L(G)^2$ since $E(H)$ is a clique in $L(G)^2$. But then $E(H')$ is a clique in $L(G[V(H')])^2 \subseteq L(G')^2$.  

Apply Lemma~\ref{stability2} to $G'=(A',B')$ and $H'$ with $n=\Delta(G)$ and $\alpha = 2\varepsilon +2\sqrt{\varepsilon} \le 4\sqrt{\varepsilon}$. Thus $G'$ contains a subgraph isomorphic to $K_{r,r}$ where $r=(1-\sqrt{2\alpha})n\geq (1-\sqrt{8\sqrt{\varepsilon}})\Delta(G) = (1-\sqrt{8}\varepsilon^{1/4})\Delta(G)$ as desired. 
\end{proof}

\end{document}